\documentclass[a4paper,reqno,final]{amsart}

\usepackage{amsmath,amssymb,amsthm}
\usepackage{fullpage,enumitem}
\usepackage{mathtools}
\mathtoolsset{showonlyrefs,showmanualtags}
\usepackage{multirow,array}
\usepackage{showkeys,version}
\usepackage{tikz} 
\usetikzlibrary{cd}
\usepackage{dynkin-diagrams}

\numberwithin{equation}{subsection}
\numberwithin{figure}{subsection}
\numberwithin{table}{subsection}
\allowdisplaybreaks
\newenvironment{Ack}%
{\par \vspace{\baselineskip}%
 \noindent \textbf{Acknowledgements.}}%
{\par \vspace{\baselineskip}}

\newtheorem{df}{Definition}[section]
\newtheorem{thm}[df]{Theorem}

\newtheorem{lem}[df]{Lemma}
\newtheorem{cor}[df]{Corollary}
\newtheorem{ex}[df]{Example}
\newtheorem{rmk}[df]{Remark}

\title{Arithmetically Cohen--Macaulay bundles on homogeneous varieties of Picard rank one} 
\author{Yusuke Nakayama}
\date{}
\begin{document}
\maketitle
\begin{abstract}
In this paper, we study arithmetically Cohen--Macaulay (ACM) bundles on homogeneous varieties $G/P$.
Indeed we characterize the homogeneous ACM bundles on $G/P$ of Picard rank one in terms of highest weights. This is a generalization of the result on $G/P$ of classical types presented by Costa and Mir\'{o}-Roig for type $A$, and Du, Fang, and Ren for types $B,C$ and $D$. As a consequence we prove that only finitely many irreducible homogeneous ACM bundles, up to twisting line bundles, exist over all such $G/P$. Moreover, we derive the list of the highest weights of the irreducible homogeneous ACM bundles on particular homogeneous varieties of exceptional types such as the Cayley Plane and the Freudenthal variety.
\end{abstract}
\section{Introduction}
Vector bundles over projective varieties have been studied for many years. For instance, Horrocks \cite{Hor} showed that vector bundles on a projective space over a field of characteristic zero split as the direct sum of line bundles if and only if they have no intermediate cohomology. Since this result was established, research on indecomposable bundles without intermediate cohomology on projective varieties has garnered considerable attention. In particular, arithmetically Cohen--Macaulay (ACM) bundles have been studied extensively.  A vector bundle, $E$, on a projective variety, $X$, is an {\it ACM\ bundle\/} if  $H^{i}(X,E(t))$ vanishes for $0<i<{\rm dim}\ X$ and for all $t\in \mathbb{Z}$ (see subsection $2.1$). ACM bundles correspond to maximal Cohen--Macaulay  modules over the associated graded ring.

ACM bundles have been studied over  particular varieties. The first nontrivial case involves two-dimensional varieties. For example, Casanellas--Hartshorne \cite{CH} proved the existence of stable Ulrich ACM bundles of an arbitrary rank on smooth cubic surfaces. This was the first example of indecomposable ACM bundles of an arbitrarily high rank on any variety except curves. Various other studies have also been conducted on this case (see \cite{BHP},\cite{CN},\cite{Fae},\cite{Wat},\cite{Yos}). In the case of three-dimensional varieties, Casnati--Faenzi--Malaspina \cite{CFM} classified all rank-two indecomposable ACM bundles on the Fano three-fold with a Picard number of three. In addition, Filip \cite{Fil} classified rank-two indecomposable ACM bundles on the general complete intersection of Calabi--Yau three-folds in projective space. Other studies have also been conducted on this topic (see \cite{BF} and \cite{RT}).

The problem of classifying ACM bundles has also been studied on homogeneous varieties. Using the Borel--Weil--Bott theorem, Costa--Mir\'{o}-Roig \cite{CM} classified irreducible homogeneous ACM bundles on Grassmannian varieties. Recently, such bundles on isotropic Grassmannians of types $B$, $C$, and $D$ were classified by Du--Fang--Ren \cite{DFR}.

The aim of this study is to classify all irreducible homogeneous ACM bundles over homogeneous varieties $X=G/P$, where $G$ is a semi-simple linear algebraic group and $P$ is a maximal parabolic subgroup. This is a generalization of the work of Costa--Mir\'{o}-Roig, and Du--Fang--Ren. We derive the necessary and sufficient conditions for an irreducible homogeneous bundle on a homogeneous variety $X$ to be an ACM bundle. The results indicate that only finitely many irreducible homogeneous ACM bundles up to twisting line bundles exist over homogeneous varieties $X$. Moreover, we derive the conditions for the highest weight for an irreducible homogeneous vector bundle on particular homogeneous varieties of an exceptional Dynkin type to be an ACM bundle. Here, it would be appropriate to mention that there is an another interesting class called Ulrich bundles. Such bundles on a homogeneous variety were studied by \cite{Fon} and \cite{LP}.

The Borel--Weil--Bott theorem is applied in the proof. To verify whether the conditions of this theorem are satisfied, we compute the pairing values of the highest weight shifted by the sum of all fundamental weights with any positive roots. 
\subsection{Statement of results}
Let $G$ be a semi-simple linear algebraic group over the complex field and $P_{\alpha_{k}}$ be the maximal parabolic subgroup associated to the simple root, $\alpha_{k}$. A vector bundle, $E$, over $G/P_{\alpha_{k}}$ is {\it homogeneous} if the action of $G$ over $G/P_{\alpha_{k}}$ can be lifted to $E$. This can be represented by $G\times_{\rho}V$, where $\rho:P_{\alpha_{k}}\to GL(V)$ is a representation of $P_{\alpha_{k}}$. If this representation is irreducible, we call $E$ an {\it irreducible\ homogeneous\ bundle\/}. We use $E_{\lambda}$ to denote the irreducible homogeneous vector bundle arising from the irreducible representation of $P_{\alpha_{k}}$ with the highest weight, $\lambda$. We define set $\Phi_{k,G}^{+}$ as follows:
$$\Phi_{k,G}^{+}:=\{\alpha\in\Phi^{+}_{G}\ |\ (\varpi_{k},\alpha)\neq0\},$$
where $\Phi_{G}^{+}$ is the set of positive roots and $(,)$ denotes the Killing form. For any $\alpha \in \Phi_{k,G}^{+}$, we denote $(\varpi_{k},\alpha)$ by $c_{\alpha,k}$. Moreover, for any irreducible homogeneous vector bundle, $E_{\lambda}$, on $G/P_{\alpha_{k}}$ with the highest weight, $\lambda$, we define set $T^{G}_{\lambda,k}$ as follows:
$$T^{G}_{\lambda,k}:=\left\{\frac{1}{c_{\alpha,k}}(\lambda+\rho,\alpha)\ |\ \alpha\in\Phi_{k,G}^{+}\right\}.$$
Note that this set essentially corresponds to the $step\ matrix$ in \cite{CM} and \cite{DFR}. For positive integer $n$, we write $[1,n]$ for the set $\{1,2,\cdots,n\}$.  We now state the main results of this study.
\begin{thm}\label{Thm} Let $E_{\lambda}$ be an initialized irreducible homogeneous vector bundle over $G/P_{\alpha_{k}}$ with highest  weight $\lambda$. Then, $E_{\lambda}$ is an ACM vector bundle if and only if $T_{\lambda,k}^{G}\cap\mathbb{Z}=[1,M^{G}_{\lambda,k}]$, where $M^{G}_{\lambda,k}$ is the maximal element in $T_{\lambda,k}^{G}$.
\end{thm}
This result was obtained by Costa and Mir\'{o}-Roig for the case in which $G/P_{\alpha_{k}}$ is Grassmannian of type $A$. Moreover, when $G/P_{\alpha_{k}}$ is an isotropic Grassmannian of type $B,\ C$, or $D$ , this result was obtained by Du, Fang, and Ren. Therefore, the novelty of the results of this study lies in proving the statement over other varieties of exceptional types.  In this paper, we present a unified proof. From this theorem, we can obtain the following corollary.
\begin{cor}\label{Cor} Only finitely many irreducible homogeneous ACM bundles up to a tensoring line bundle exist on $G/P_{\alpha_{k}}$.
\end{cor}
In addition, based on the aforementioned theorem, we can determine the conditions for the highest weight for an irreducible homogeneous vector bundle on particular homogeneous varieties of exceptional types (e.g., Cayley Plane and Freudenthal variety) to be an ACM bundle (see Section $4$).

\begin{Ack}
I am grateful to Takeshi Ikeda for their helpful advice and comments, to Hajime Kaji for beneficial comments, and to Satoshi Naito and Takafumi Kouno for useful advice. I would like to thank Editage (www.editage.com) for English language editing.
\end{Ack}
\section{Preliminaries}
We begin this section by reviewing some definitions and introducing the notations. All algebraic varieties in this study are defined over the field of complex numbers, $\mathbb{C}$.
\subsection{ACM bundles}
We introduce ACM bundles on a projective variety as this is the main focus of this study.
\begin{df} Let $\iota:X \subset \mathbb{P}^{N}$ be a projective variety with $\mathcal{O}_{X}(1):=\iota^{*}\mathcal{O}_{\mathbb{P}^{N}}(1)$. A vector bundle, $E$, over $X$ is called {\it arithmetically\ Cohen\ Macauley} (ACM)  if 
$$H^{i}(X,E(t))=0,\ {\rm where}\ E(t):=E\otimes_{\mathcal{O}_{X}}\mathcal{O}_{X}(t),\ {\rm for\ all}\ i=1,\cdots,{\rm dim}\ X-1\ {\rm and}\ t \in \mathbb{Z}.$$  
\end{df}
By definition, $E$ being an ACM bundle over a projective variety, $X$, is equivalent to $E(t)$ being an ACM bundle over $X$ for $t \in \mathbb{Z}$. In this paper, we only consider the case in which $X$ is a homogeneous variety of Picard rank one. In particular, we aim to classify all irreducible homogeneous ACM bundles over homogeneous varieties of Picard rank one.
\subsection{Homogeneous vector bundles}
Let $G$ be a semi-simple linear algebraic group over the complex field $\mathbb{C}$ and $T\subset G$ be a maximal torus. We set $\mathfrak{g}:=$Lie$(G)$ and $\mathfrak{h}:=$Lie$(T)$. Let $\Phi$ be the set of roots associated with pair $(\mathfrak{g},\mathfrak{h})$. We fix set $\Delta:=\{\alpha_{1},\cdots,\alpha_{n}\}$ of simple roots of $\Phi$. Let $\Phi^{+}$ be the set of positive roots. The {\it weight lattice} $\Lambda$ is the set of all linear functions $\lambda:\mathfrak{h}\to\mathbb{C}$ for which $\frac{2(\lambda,\alpha)}{(\alpha,\alpha)}\in \mathbb{Z}$ for any $\alpha \in \Phi$, where $(,)$ denotes the Killing form. We define $\lambda \in \Lambda$ to be {\it dominant\/} if all integers $\frac{2(\lambda,\alpha)}{(\alpha,\alpha)}\ (\alpha \in \Phi)$ are nonnegative and {\it strongly\ dominant\/} if they are positive. Let $\varpi_{1},\cdots,\varpi_{n}\in \Lambda$ be the {\it fundamental weights\/}, i.e., $\frac{2(\varpi_{i},\alpha_{j})}{(\alpha_{j},\alpha_{j})}=\delta_{i,j}$. From this definition, $\lambda=\sum_{i=1}^{n}a_{i}\varpi_{i}$ is a dominant weight if and only if $a_{i}\geq 0$ and a strongly dominant weight if and only if $a_{i}>0$ for $1\leq i \leq n$. For each dominant $\lambda\in \Lambda$, the irreducible highest weight representation of  $\mathfrak{g}$ with highest weight $\lambda$ is denoted by $V_\lambda$.

For $\alpha_{k}\in \Delta$, let $\Phi^{-}(\alpha_{k})$ be a set consisting of negative roots $\alpha$ such that $\alpha=\sum_{\alpha_{i}\neq\alpha_{k}}a_{i}\alpha_{i}$. Let
$$\mathfrak{p}_{\alpha_{k}}:=\mathfrak{h}{\huge \oplus}(\oplus_{\alpha\in\Phi^{+}}\mathfrak{g}_{\alpha})\oplus(\oplus_{\alpha\in\Phi^{-}(\alpha_{k})}\mathfrak{g}_{\alpha}),$$
and $P_{\alpha_{k}}$ be a subgroup of $G$ such that the Lie algebra of $P_{\alpha_{k}}$ is $\mathfrak{p}_{\alpha_{k}}$, where each $\mathfrak{g}_{\alpha}$ is a one-dimensional eigenspace with respect to the adjoint action of $\mathfrak{h}$. Note that since the dimension of the Lie algebra of $P_{\alpha_{k}}$ is equal to the number of positive roots $\alpha$ such that $(\alpha,\varpi_{k})=0$, the cardinality of $\Phi_{k,G}^{+}$ is equal to the dimension of the homogeneous variety $G/P_{\alpha_{k}}$. 

Next, we consider vector bundles on $G/P$. In particular, we introduce an important class of vector bundles. A vector bundle, $E$, with fiber $V$ over $G/P$ is {\it homogeneous\/} if  there exists representation $\rho:P\to GL(V)$ such that $E\cong E_{\rho}$, where $E_{\rho}$ is a vector bundle over $G/P$ with fiber $V$ originating from the principal bundle, $G\to G/P$, via $\rho$. Let $E$ be a homogeneous vector bundle on $G/P$. If representation $\rho:P\to GL(V)$ is irreducible, we call $E$ an $irreducible\ homogeneous\ vector\ bundle$. Any homogeneous vector bundle can be decomposed into irreducible homogeneous vector bundles . Therefore, we only classify irreducible homogeneous vector bundles. 

We describe all irreducible representations of a parabolic subgroup, $P_{\alpha_{k}}$. Let $\alpha_{k}$ be a simple root and $\varpi_{k}$ be the corresponding fundamental weight. Let $S_{P_{\alpha_{k}}}$ be the semi-simple part of $P_{\alpha_{k}}$. Then, all irreducible representations of $P_{\alpha_{k}}$ are
$$V\otimes L_{\varpi_{k}}^{t},$$
where $V$ is an irreducible representation of $S_{P_{\alpha_{k}}}$, $t\in \mathbb{Z}$, and $L_{\varpi_{k}}$ is a one-dimensional representation with weight $\varpi_{k}$. If $\lambda$ is the highest weight of irreducible representation $V$ of $S_{P}$, $\lambda+t\varpi_{k}$ is the highest weight of irreducible representation $V\otimes L_{\varpi_{k}}^{t}$ of $P_{\alpha_{k}}$. 

In this study, $E_{\lambda}$ denotes the irreducible homogeneous vector bundle arising from the irreducible representation of $P_{\alpha_{k}}$ with highest weight $\lambda$.\ As any irreducible representation of a semi-simple Lie group is determined by its highest weight, if $E_{\lambda}$ is an irreducible homogeneous vector bundle on $G/P_{\alpha_{k}}$ with $\lambda=\sum a_{i}\varpi_{i}$, then $a_{i}\geq 0$ for $i\neq k$.

Finally, we introduce the following definition.
\begin{df} A vector bundle, $E$, on $G/P_{\alpha_{k}}$ is called {\it initialized\/} if 
$$H^{0}(G/P_{\alpha_{k}},E(-1))=0\ \ {\rm and}\ \ H^{0}(G/P_{\alpha_{k}},E)\neq 0.$$ 
\end{df}
For an initialized homogeneous vector bundle on a homogeneous variety, the following result is known (see \cite{DFR}).
\begin{lem}\label{lem} Let $E_{\lambda}$ be an initialized homogeneous vector bundle on $G/P_{\alpha_{k}}$ with $\lambda=\sum_{i=1}^{n}a_{i}\varpi_{i}$. Then, $a_{k}=0$.
\end{lem}
\section{Proof of Theorem \ref{Thm}}
\ In this section, we prove Theorem \ref{Thm}. First, some definitions are presented.
\begin{df}\ Let $\lambda$ be a weight.\\
(i)\ $\lambda$ is called {\it singular\/} if there exists $\alpha \in \Phi^{+}$ such that $(\lambda,\alpha)=0$.\\
(ii)\ $\lambda$ is called {\it regular\ of\ index\/} $p$ if it is not singular and if there are exactly $p$ roots $\alpha_{1},\cdots,\alpha_{p} \in \Phi^{+}$ such that $(\lambda,\alpha_{i})<0$.
\end{df}
\ The following lemma is crucial  to the proof of the main result.
\begin{lem}\label{Lem2} Let $E_{\lambda}$ be an irreducible homogeneous vector bundle on a homogeneous variety $G/P_{\alpha_{k}}$ with the highest weight $\lambda$. Then, $E_{\lambda}$ is an ACM bundle if and only if one of the following conditions holds for each $t \in \mathbb{Z}$: \\ 
\ \ \ $1)\ \lambda+\rho-t\varpi_{k}$ is regular  of index $0$;\\
\ \ \ $2)\ \lambda+\rho-t\varpi_{k}$ is regular of index dim $G/P_{\alpha_{k}}$; and\\
\ \ \ $3)\ \lambda+\rho-t\varpi_{k}$ is singular.
\end{lem}
\begin{proof}
The result follows from the definition of ACM bundles and the following theorem.
\end{proof}
\begin{thm}[Borel-Weil-Bott] Let $E_{\lambda}$ be an irreducible homogeneous vector bundle over $G/P$.\\
(i) If $\lambda+\rho$ is singular, then
$$H^{i}(G/P,E_{\lambda})=0,\ \forall i.$$
(ii) If $\lambda+\rho$ is regular of index $p$, then
$$H^{i}(G/P,E_{\lambda})=0,\ \forall i\neq p$$
and
$$H^{p}(G/P,E_{\lambda})=V_{w(\lambda+\rho)-\rho},$$
where we write the sum of all fundamental weights for $\rho$;  $w(\lambda+\rho)$ is the unique element of the fundamental Weyl Chamber of $G$, which is congruent to $\lambda+\rho$ under the action of the Weyl group; and $V_{w(\lambda+\rho)-\rho}$ is an irreducible representation of $G$.
\end{thm}
\begin{rmk}\ In \cite{Dem}, a parabolic subgroup is stated in the case of Borel subgroup. However, it is known that it is sufficient to consider such cases (See \cite{Ott}).
\end{rmk}
\begin{proof}[Proof of Theorem \ref{Thm}] We examine these conditions by using Lemma \ref{Lem2}.
\begin{lem}\ $\lambda+\rho-t\varpi_{k}$ is regular of index $0$ if and only if $t<1$.
\end{lem}
\begin{proof}
Suppose that $\lambda+\rho-t\varpi_{k}$ is regular of index $0$. Then, $\lambda+\rho-t\varpi_{k}=\sum_{i\neq k}(1+a_{i})\varpi_{i}+(1-t+a_{k})\varpi_{k}$ is strongly dominant, i.e., $a_{i}+1>0$ and $1+a_{k}-t>0$. Now, as $a_{i}\geq0$ for $i \neq k$, we only consider the case $1+a_{k}-t>0$. By initializing (Lemma $2.10$), we obtain $a_{k}=0$. Therefore, $t<1$. Conversely, suppose $t<1$. Using the same argument, $\lambda+\rho-t\varpi_{k}$ is regular of index $0$. Hence, $\lambda+\rho-t\varpi_{k}$ is regular of index $0$ if and only if $t<1$.
\end{proof}
Next, we consider case $(2)$.
\begin{lem}\ $\lambda+\rho-t\varpi_{k}$ is regular of index dim $G/P_{\alpha_{k}}$ if and only if $t>M^{G}_{\lambda,k}$.
\end{lem}
\begin{proof}
It suffices to prove that $\lambda+\rho-t\varpi_{k}$ is regular of index $|\Phi_{k,G}^{+}|$ if and only if $t>M^{G}_{\lambda,k}$. Note that $(\lambda+\rho-t\varpi_{k},\alpha)=(\lambda+\rho,\alpha)-c_{\alpha,k}\cdot t$ for $\alpha\in\Phi_{k,G}^{+}$. Suppose that $t>M^{G}_{\lambda,k}$. By the definition of $M^{G}_{\lambda,k}$, the pairing value $(\lambda+\rho-t\varpi_{k},\alpha)$ is negative for all $\alpha \in \Phi_{k,G}^{+}$. As the cardinality of $T_{\lambda,k}^{G}$ is equal to the cardinality of $\Phi_{k,G}^{+}$, $\lambda+\rho-t\varpi_{k}$ is regular of index $|\Phi_{k,G}^{+}|$. Conversely, if $\lambda+\rho-t\varpi_{k}$ is regular of index $|\Phi_{k,G}^{+}|$, then the pairing value $(\lambda+\rho-t\varpi_{k},\alpha)$ is negative for all $\alpha \in \Phi_{k,G}^{+}$. In particular, $t>M^{G}_{\lambda,k}$.
\end{proof}
Therefore, $E_{\lambda}$ being an ACM bundle is equivalent to $\lambda+\rho-t\varpi_{k}$ being singular for all integers $t \in [1,M^{G}_{\lambda,k}]$ (i.e., there exists a positive root $\alpha$ such that $(\lambda+\rho-t\varpi_{k},\alpha)=0$ for all $t \in [1,M^{G}_{\lambda,k}]$). If $\alpha$ is in $\Phi^{+}_{G}\setminus\Phi_{k,G}^{+}$, then the pairing value $(\lambda+\rho-t\varpi_{k},\alpha)=(\lambda+\rho,\alpha)$. As $a_{i}\geq0$ and $(\rho,\alpha)>0$, this is positive. Hence, we consider only the case in which $\alpha$ is in $\Phi_{k,G}^{+}$.
\begin{lem}\ $\lambda+\rho-t\varpi_{k}$ is singular for any $t \in [1,M^{G}_{\lambda,k}]$ if and only if $t\in T_{\lambda,k}^{G}$ for all $t \in [1,M^{G}_{\lambda,k}]$.
\end{lem}
\begin{proof} If $\lambda+\rho-t\varpi_{k}$ is singular for any $t \in [1,M^{G}_{\lambda,k}]$, there exists a positive root $\alpha$ in $\Phi_{k,G}^{+}$ such that
$$0=(\lambda+\rho-t\varpi_{k},\alpha)=(\lambda+\rho,\alpha)-c_{\alpha,k}\cdot t.$$
As $\alpha$ is in $\Phi_{k,G}^{+}$, $c_{\alpha,k}$ is not equal to zero. Thus, $t=\frac{1}{c_{\alpha,k}}(\lambda+\rho,\alpha)$. By the definition of $T_{\lambda,k}^{G}$, $t$ lies is in it. Conversely, if $t$ is in $T_{\lambda,k}^{G}$ for any $t \in [1,M^{G}_{\lambda,k}]$, there exists a positive root $\alpha$ in $\Phi_{k,G}^{+}$ such that $t=\frac{1}{c_{\alpha,k}}(\lambda+\rho,\alpha)$. Then,
$$0=(\lambda+\rho,\alpha)-c_{\alpha,k}\cdot t=(\lambda+\rho-t\varpi_{k},\alpha).$$
Therefore, $\lambda+\rho-t\varpi_{k}$ is singular for any $t \in [1,M^{G}_{\lambda,k}]$.
\end{proof}
Therefore, in summary, $E_{\lambda}$ is an ACM vector bundle if and only if $T_{\lambda,k}^{G}\cap\mathbb{Z}=[1,M^{G}_{\lambda,k}]$.
\end{proof}

From now on, we denote $n_{l}:=\#\{\alpha\in \Phi_{k,G}^{+} |\ \frac{1}{c_{\alpha,k}}(\lambda+\rho,\alpha)=l\}$. Then the main theorem can be rephrased as follows. Let $E_{\lambda}$ be an initialized irreducible homogeneous vector bundle over $G/P_{\alpha_{k}}$ with highest  weight $\lambda$. Then, $E_{\lambda}$ is an ACM vector bundle if and only if $n_{l}\geq1$ for any integer $l\in[1,M^{G}_{\lambda,k}]$.

In the following subsections, we present some observations on homogeneous varieties of different types. In particular, we calculate the pairing values of $\lambda+\rho$ with all positive roots and $M_{\lambda,k}^{G}$. 
\subsection{type $E_{6}$}
In this subsection, we focus on homogeneous varieties of type $E_{6}$. We present some notations used in the following three subsections on homogeneous varieties of types $E_{6},\ E_{7}$, and $E_{8}$.

Let $G$ be a simply connected simple Lie group with a Dynkin diagram of type $E_{6},\ E_{7}$, or $E_{8}$ as follows.
\begin{center}
$E_{6}$:\dynkin[label]{E}{6}\ \ \ \ \ \ $E_{7}$:\dynkin[label]{E}{7}\ \ \ \ \ \ $E_{8}$:\dynkin[label]{E}{8}
\end{center}
$\epsilon_{1},\cdots,\epsilon_{n}$ denotes the usual standard basis for $\mathbb{R}^{n}$. The $\mathbb{Z}$-span of this basis is a lattice denoted by $I_{n}$. Let $I^{\prime}_{8}=I_{8}+\mathbb{Z}((\epsilon_{1}+\cdots+\epsilon_{8})/2)$ and $L$ be a subgroup of $I^{\prime}_{8}$ consisting of all elements $\sum c_{i}\epsilon_{i}+\frac{c}{2}(\epsilon_{1}+\cdots+\epsilon_{8})$ for which $\sum c_{i}$ is an even integer. We define $\Phi_{E_{8}}=\{\alpha \in L\ |\ (\alpha,\alpha)=2\}$. $\Phi_{E_{8}}$ consists of the vectors $\pm(\epsilon_{i}\pm\epsilon_{j}),\ i\neq j,$ along with the lesser ones $\frac{1}{2}\sum_{i=1}^{8}(-1)^{k(i)}\epsilon_{i}$ (where $k(i)=0,1$ sum to an even integer). As a base, we take
$$\Delta_{E_{8}}=\{\alpha_{1}:=\frac{1}{2}(\epsilon_{1}-\epsilon_{2}-\epsilon_{3}-\epsilon_{4}-\epsilon_{5}-\epsilon_{6}-\epsilon_{7}+\epsilon_{8}),\ \alpha_{2}:=\epsilon_{1}+\epsilon_{2},\ \alpha_{i}:=\epsilon_{i-1}-\epsilon_{i-2}\ (3\leq i\leq8)\}.$$
$\alpha_{(i_{1},\cdots,i_{j})}\ (1\leq i_{1}<\cdots<i_{j} \leq 7;\ j\ {\rm is\ even.})$ denotes an element in $\Phi_{E_{8}}$ (i.e., the root) such that the coefficients of $\epsilon_{i_{l}}\ (l=1,\cdots,j)$ are $-\frac{1}{2}$, and the remaining elements are $\frac{1}{2}$. Moreover, let $\beta_{i}\ (1\leq i\leq 7)$ denote an element in $\Phi_{E_{8}}$ such that the coefficients of $\epsilon_{i}$ are $\frac{1}{2}$ and those of the remaining are $-\frac{1}{2}$.

Let $V_{6}$ denote the span $\{\alpha_{1},\cdots,\alpha_{6}\}_{\mathbb{R}{\rm -lin}} \subset \mathbb{R}^{8}$. We define $\Phi_{E_{6}}=\Phi_{E_{8}}\cap V_{6}$. $\Phi_{E_{6}}$ consists of the vectors $\pm(\epsilon_{i}\pm\epsilon_{j}),\ (1\leq i<j\leq 5)$ along with the lesser ones $\frac{1}{2}(\epsilon_{8}-\epsilon_{7}-\epsilon_{6}+\sum_{i=1}^{5}(-1)^{k(i)}\epsilon_{i})$ (where $k(i)=0,1$ sum to an odd integer).
As a base, we take
$$\Delta_{E_{6}}:=\{\alpha_{1},\cdots,\alpha_{6}\}.$$
The fundamental weights are as follows.
\begin{eqnarray*}
\varpi_{1}&=&-\frac{2}{3}\epsilon_{6}-\frac{2}{3}\epsilon_{7}+\frac{2}{3}\epsilon_{8}\\
\varpi_{2}&=&\frac{1}{2}(\epsilon_{1}+\epsilon_{2}+\epsilon_{3}+\epsilon_{4}+\epsilon_{5}-\epsilon_{6}-\epsilon_{7}+\epsilon_{8})\\
\varpi_{3}&=&\frac{1}{2}(-\epsilon_{1}+\epsilon_{2}+\epsilon_{3}+\epsilon_{4}+\epsilon_{5})+\frac{5}{6}(-\epsilon_{6}-\epsilon_{7}+\epsilon_{8})\\
\varpi_{4}&=&\epsilon_{3}+\epsilon_{4}+\epsilon_{5}-\epsilon_{6}-\epsilon_{7}+\epsilon_{8}\\
\varpi_{5}&=&\epsilon_{4}+\epsilon_{5}+\frac{2}{3}(-\epsilon_{6}-\epsilon_{7}+\epsilon_{8})\\
\varpi_{6}&=&\epsilon_{5}+\frac{1}{3}(-\epsilon_{6}-\epsilon_{7}+\epsilon_{8})
\end{eqnarray*} 
The dimension of a homogeneous variety $E_{6}/P_{\alpha_{k}}$ is
$$\left\{
\begin{array}{ll}
16 & {\rm if}\ k=1,6\\ 
21& {\rm if}\ k=2\\ 
25& {\rm if}\ k=3,5\\ 
29& {\rm if}\ k=4\\ 
\end{array}
.\right.
$$
The set of positive roots, $\Phi_{E_{6}}^{+}$, is
$$\{\pm\epsilon_{i}+\epsilon_{j}\}_{1\leq i<j\leq 5}\cup\{\alpha_{i,j,6,7}\}_{1\leq i<j\leq 5}\cup\{\alpha_{6,7}\}\cup\{\beta_{i}\}_{1\leq i\leq5}.$$
The total number of positive roots is $36$. We now calculate the pairing values of $\lambda+\rho$ with all positive roots. For any $1\leq i<j\leq5$ and $1\leq l\leq5$, we have
\begin{eqnarray*}
(\lambda+\rho,\epsilon_{i}+\epsilon_{j})&=&\sum_{s=1}^{i}a_{1+s}+\sum_{t=1}^{j-2}a_{3+t}+i+j-2\\
(\lambda+\rho,-\epsilon_{i}+\epsilon_{j})&=&\sum_{t=1}^{j-i}a_{t+1+i}+j-i\\
(\lambda+\rho,\alpha_{6,7})&=&a_{1}+ 2a_{2}+ 2a_{3}+ 3a_{4}+ 2a_{5}+ a_{6}+11\\
(\lambda+\rho,\alpha_{i,j,6,7})&=&a_{1}+ a_{2}+ 2a_{3}+ 3a_{4}+ 2a_{5}+ a_{6}-(\sum_{s=1}^{i-1}a_{s+2}+\sum_{t=1}^{j-2}a_{3+t})+13-i-j\\
(\lambda+\rho,\beta_{l})&=&a_{1}+ a_{3}+ a_{4}+ a_{5}+a_{6}-\sum_{s=1}^{5-l}a_{7-s}+l.
\end{eqnarray*}
Using these calculations and by calculating the pairing values $(\varpi_{k},\alpha)$ for any $\alpha\in\Phi_{k,E_{6}}^{+}$, we can verify that $M_{\lambda,k}^{E_{6}}$ is given by
$$\left\{
\begin{array}{ll}
a_{1}+2a_{2}+2a_{3}+3a_{4}+2a_{5}+a_{6}+11& {\rm if}\ k=1,6\\ 
a_{1}+a_{2}+2a_{3}+3a_{4}+2a_{5}+a_{6}+10& {\rm if}\ k=2\\ 
a_{1}+a_{2}+2a_{3}+2a_{4}+a_{5}+a_{6}+8& {\rm if}\ k=3,5\\ 
a_{1}+a_{2}+a_{3}+a_{4}+a_{5}+a_{6}+6 & {\rm if}\ k=4 
\end{array}
.\right.
$$
Finally, we consider an example of Theorem \ref{Thm} for varieties of type $E_{6}$.
\begin{ex}\ Let $E_{\lambda}$ and $E_{\mu}$ be initialized irreducible homogeneous vector bundles on $E_{6}/P(\alpha_{2})$ with $\lambda=2\varpi_{1}+\varpi_{3}$ and $\mu=\varpi_{4}+\varpi_{5}$. Then, $\Phi_{2,E_{6}}^{+}$ consists of $\epsilon_{i}+\epsilon_{j}$, $\alpha_{i,j,6,7}$, or $\alpha_{6,7}$ for all $1\leq i<j\leq5$. We introduce the following matrix for convenience.
\[\frac{1}{c_{\alpha,k}}
\begin{pmatrix}
0&(\lambda+\rho,\epsilon_{1}+\epsilon_{2}) &(\lambda+\rho,\epsilon_{1}+\epsilon_{3}) &(\lambda+\rho,\epsilon_{1}+\epsilon_{4}) &(\lambda+\rho,\epsilon_{1}+\epsilon_{5})\\
 (\lambda+\rho,\alpha_{1,2 , 6,7})&0&(\lambda+\rho,\epsilon_{2}+\epsilon_{3}) &(\lambda+\rho,\epsilon_{2}+\epsilon_{4}) &(\lambda+\rho,\epsilon_{2}+\epsilon_{5})\\
 (\lambda+\rho,\alpha_{1,3 , 6,7}) &(\lambda+\rho,\alpha_{2,3 , 6,7}) &0&(\lambda+\rho,\epsilon_{3}+\epsilon_{4}) &(\lambda+\rho,\epsilon_{3}+\epsilon_{5})\\
 (\lambda+\rho,\alpha_{1,4 , 6,7})  & (\lambda+\rho,\alpha_{2,4 , 6,7})& (\lambda+\rho,\alpha_{3,4 , 6,7})&0&(\lambda+\rho,\epsilon_{4}+\epsilon_{5})\\
 (\lambda+\rho,\alpha_{1,5 , 6,7})&(\lambda+\rho,\alpha_{2, 5, 6,7})& (\lambda+\rho,\alpha_{3, 5, 6,7})&(\lambda+\rho,\alpha_{4,5 , 6,7})&(\lambda+\rho,\alpha_{6,7})
\end{pmatrix}
\]
where $\alpha$ is denoted by a positive root in $\Phi_{2,E_{6}}^{+}$ contained in the pairing of each matrix element. We identify this matrix with the set $T_{\lambda,2}^{E_{6}}$. We consider $T_{\mu,2}^{E_{6}}$  in the same way. Then, we have
\[
T_{\lambda,2}^{E_{6}}=
\begin{pmatrix}
0&1&2&3&4\\
14&0&4&5&6\\
13&11&0&6&7\\
12&10&9&0&8\\
11&9&8&7&\frac{15}{2}
\end{pmatrix}
,\ T_{\mu,2}^{E_{6}}=
\begin{pmatrix}
0&1&3&5&6\\
15&0&4&6&7\\
13&12&0&8&9\\
11&10&8&0&11\\
10&9&7&5&8
\end{pmatrix}
.\]
For $E_{\lambda}$, we can verify that $M_{\lambda,2}^{E_{6}}=14$ and $n_{l}\geq1$ for any $l\in[1,M_{\lambda,2}^{E_{6}}]$. Therefore, by the aforementioned theorem, $E_{\lambda}$ is an ACM bundle. In addition, we can verify that $M_{\mu,2}^{E_{6}}=15$ and $n_{2}=n_{14}=0$. Therefore, by the aforementioned theorem, $E_{\mu}$ is not an ACM bundle.
\end{ex}
\subsection{type $E_{7}$}
Let $V_{7}$ denote the span $\{\alpha_{1},\cdots,\alpha_{7}\}_{\mathbb{R}{\rm -lin}} \subset \mathbb{R}^{8}$. We define $\Phi_{E_{7}}=\Phi_{E_{8}}\cap V_{7}$. As a set, $\Phi_{E_{7}}$ is
$$\{\pm\epsilon_{i}\pm\epsilon_{j}\}_{1\leq i<j\leq 6}\cup\{\pm(\epsilon_{7}-\epsilon_{8})\}\cup\{\pm\frac{1}{2}(\epsilon_{8}-\epsilon_{7}+\sum_{i=1}^{6}\pm\epsilon_{i})\ |\ {\rm The\ number\ of\ minus\ is\ odd}\}.$$ 
As a base, we take
$$\Delta_{E_{7}}=\{\alpha_{1},\cdots,\alpha_{7}\}.$$
The fundamental weights are as follows:
\begin{eqnarray*}
\varpi_{1}&=&-\epsilon_{7}+\epsilon_{8}\\
\varpi_{2}&=&\frac{1}{2}(\epsilon_{1}+\epsilon_{2}+\epsilon_{3}+\epsilon_{4}+\epsilon_{5}+\epsilon_{6})-\epsilon_{7}+\epsilon_{8}\\
\varpi_{3}&=&\frac{1}{2}(-\epsilon_{1}+\epsilon_{2}+\epsilon_{3}+\epsilon_{4}+\epsilon_{5}+\epsilon_{6}-3\epsilon_{7}+3\epsilon_{8})\\
\varpi_{4}&=&\epsilon_{3}+\epsilon_{4}+\epsilon_{5}+\epsilon_{6}-2\epsilon_{7}+2\epsilon_{8}\\
\varpi_{5}&=&\epsilon_{4}+\epsilon_{5}+\epsilon_{6}-\frac{3}{2}(\epsilon_{7}-\epsilon_{8})\\
\varpi_{6}&=&\epsilon_{5}+\epsilon_{6}-\epsilon_{7}+\epsilon_{8}\\
\varpi_{7}&=&\epsilon_{6}-\frac{1}{2}(\epsilon_{7}-\epsilon_{8}).
\end{eqnarray*}
The set of positive roots, $\Phi_{E_{7}}^{+}$, is
$$\{\pm\epsilon_{i}+\epsilon_{j}\}_{1\leq i<j\leq 6}\cup\{-\epsilon_{7}+\epsilon_{8}\}\cup\{\alpha_{l,i,j,7}\}_{1\leq l< i<j\leq 6}\cup\{\alpha_{i,7}\}_{1\leq i\leq6}\cup\{\beta_{i}\}_{1\leq i\leq6}.$$
The total number of positive roots is $63$. The dimension of a homogeneous variety $E_{7}/P_{\alpha_{k}}$ is
$$\left\{
\begin{array}{ll}
33 & {\rm if}\ k=1\\ 
42& {\rm if}\ k=2\\ 
47& {\rm if}\ k=3\\ 
53& {\rm if}\ k=4\\ 
50& {\rm if}\ k=5\\ 
42& {\rm if}\ k=6\\ 
27& {\rm if}\ k=7\\ 
\end{array}
.\right.
$$
We now calculate the pairing values of $\lambda+\rho$ with all positive roots. For any $1\leq i<j\leq6$, $1\leq i_{1}< i_{2}< i_{3} \leq6$, and $1\leq l\leq6$, we have
\begin{eqnarray*}
(\lambda+\rho,\epsilon_{i}+\epsilon_{j})&=&\sum_{s=1}^{i}a_{1+s}+\sum_{t=1}^{j-2}a_{3+t}+i+j-2\\
(\lambda+\rho,-\epsilon_{i}+\epsilon_{j})&=&\sum_{t=1}^{j-i}a_{t+1+i}+j-i\\
(\lambda+\rho,-\epsilon_{7}+\epsilon_{8})&=& 2a_{1}+2a_{2}+3a_{3}+4a_{4}+3a_{5}+2a_{6}+a_{7}+17\\
(\lambda+\rho,\beta_{l})&=&a_{1}+ a_{3}+ a_{4}+ a_{5}+a_{6}+a_{7}-\sum_{s=1}^{6-l}a_{8-s}+l\\
(\lambda+\rho,\alpha_{l,7})&=&a_{1}+ 2a_{2}+ 3a_{3}+ 4a_{4}+ 3a_{5}+ 2a_{6}+ a_{7}-\sum_{s=1}^{l-1}a_{s+2}+17-l\\
(\lambda+\rho,\alpha_{i_{1},i_{2},i_{3},7})&=&a_{1}+ a_{2}+ 2a_{3}+ 3a_{4}+ 3a_{5}+ 2a_{6}+ a_{7}-\sum_{r=1}^{i_{1}-1}(a_{r+1}+a_{r+2}+a_{r+3})\\
&&-(\sum_{s=1}^{i_{2}-i_{1}-1}a_{s+i_{1}+2}+\sum_{t=1}^{i_{3}-i_{1}-2}a_{t+i_{1}+3})+19-i_{1}-i_{2}-i_{3}.
\end{eqnarray*}
Using these calculations and by calculating $(\varpi_{k},\alpha)$ for any $\alpha\in\Phi_{k,E_{7}}^{+}$, we can verify that $M_{\lambda,k}^{E_{7}}$ is equal to
$$\left\{
\begin{array}{ll}
a_{1}+2a_{2}+3a_{3}+4a_{4}+3a_{5}+2a_{6}+a_{7}+16& {\rm if}\ k=1\\ 
a_{1}+a_{2}+2a_{3}+3a_{4}+3a_{5}+2a_{6}+a_{7}+13& {\rm if}\ k=2\\ 
a_{1}+a_{2}+a_{3}+2a_{4}+2a_{5}+2a_{6}+a_{7}+10& {\rm if}\ k=3\\ 
a_{1}+a_{2}+a_{3}+a_{4}+a_{5}+a_{6}+a_{7}+7& {\rm if}\ k=4\\
a_{1}+a_{2}+2a_{3}+2a_{4}+a_{5}+a_{6}+a_{7}+9& {\rm if}\ k=5\\
a_{1}+2a_{2}+2a_{3}+3a_{4}+2a_{5}+a_{6}+a_{7}+12& {\rm if}\ k=6\\
2a_{1}+2a_{2}+3a_{3}+4a_{4}+3a_{5}+2a_{6}+a_{7}+17 & {\rm if}\ k=7
\end{array}
.\right.
$$
\subsection{type $E_{8}$}
Similar to the cases described in previous subsections, we take the base to be 
$$\Delta_{E_{8}}=\{\alpha_{1}:=\frac{1}{2}(\epsilon_{1}-\epsilon_{2}-\epsilon_{3}-\epsilon_{4}-\epsilon_{5}-\epsilon_{6}-\epsilon_{7}+\epsilon_{8}),\ \alpha_{2}:=\epsilon_{1}+\epsilon_{2},\ \alpha_{i}:=\epsilon_{i-1}-\epsilon_{i-2}\ (3\leq i\leq8)\}.$$
The fundamental weights are as follows.
\begin{eqnarray*}
\varpi_{1}&=&2\epsilon_{8}\\
\varpi_{2}&=&\frac{1}{2}(\epsilon_{1}+\epsilon_{2}+\epsilon_{3}+\epsilon_{4}+\epsilon_{5}+\epsilon_{6}+\epsilon_{7}+5\epsilon_{8})\\
\varpi_{3}&=&\frac{1}{2}(-\epsilon_{1}+\epsilon_{2}+\epsilon_{3}+\epsilon_{4}+\epsilon_{5}+\epsilon_{6}+\epsilon_{7}+7\epsilon_{8})\\
\varpi_{4}&=&\epsilon_{3}+\epsilon_{4}+\epsilon_{5}+\epsilon_{6}+\epsilon_{7}+5\epsilon_{8}\\
\varpi_{5}&=&\epsilon_{4}+\epsilon_{5}+\epsilon_{6}+\epsilon_{7}+4\epsilon_{8}\\
\varpi_{6}&=&\epsilon_{5}+\epsilon_{6}+\epsilon_{7}+3\epsilon_{8}\\
\varpi_{7}&=&\epsilon_{6}+\epsilon_{7}+2\epsilon_{8}\\
\varpi_{8}&=&\epsilon_{7}+\epsilon_{8}
\end{eqnarray*}
The set of positive roots, $\Phi_{E_{8}}^{+}$, is
$$\{\pm\epsilon_{i}+\epsilon_{j}\}_{1\leq i<j\leq 8}\cup\{\alpha_{0,0},\ \alpha_{i_{1},i_{2}}\}_{1\leq i_{1}< i_{2}\leq 7}\cup\{\alpha_{i_{1},i_{2},i_{3},i_{4}}\}_{1\leq i_{1}< i_{2}<i_{3}<i_{4}\leq 7}\cup\{\beta_{i}\}_{1\leq i\leq6},$$
where $\alpha_{0,0}$ is the  root in $\Phi_{E_{8}}$ such that all coefficients of $\epsilon_{i}$ are $\frac{1}{2}$ for $1\leq i \leq 8$. The total number of positive roots is $120$. The dimension of homogeneous variety $E_{8}/P_{\alpha_{k}}$ is
$$\left\{
\begin{array}{ll}
78& {\rm if}\ k=1\\ 
92& {\rm if}\ k=2\\ 
98& {\rm if}\ k=3\\ 
106& {\rm if}\ k=4\\ 
104& {\rm if}\ k=5\\ 
97& {\rm if}\ k=6\\ 
83& {\rm if}\ k=7\\ 
57& {\rm if}\ k=8\\ 
\end{array}
.\right.
$$
We now calculate the pairing values of $\lambda+\rho$ with all positive roots. For any $1\leq i<j\leq7$, $1\leq i_{1}< i_{2}< i_{3} <i_{4}\leq7$, and $1\leq l\leq7$, we have
\begin{eqnarray*}
(\lambda+\rho,\epsilon_{i}+\epsilon_{j})&=&\sum_{s=1}^{i}a_{1+s}+\sum_{t=1}^{j-2}a_{3+t}+i+j-2\\
(\lambda+\rho,\epsilon_{l}+\epsilon_{8})&=& 2a_{1}+3a_{2}+3a_{3}+5a_{4}+4a_{5}+3a_{6}+2a_{7}+a_{8}+\sum_{s=1}^{l-1}a_{s+2}+l+22\\
(\lambda+\rho,-\epsilon_{i}+\epsilon_{j})&=&\sum_{t=1}^{j-i}a_{t+1+i}+j-i\\
(\lambda+\rho,-\epsilon_{l}+\epsilon_{8})&=&2a_{1}+3a_{2}+3a_{3}+5a_{4}+4a_{5}+3a_{6}+2a_{7}+a_{8}-\sum_{t=3}^{l+1}a_{t}+24-l \\
(\lambda+\rho,\beta_{l})&=&a_{1}+ a_{3}+ a_{4}+ a_{5}+a_{6}+a_{7}+a_{8}-\sum_{s=1}^{7-l}a_{9-s}+l\\
(\lambda+\rho,\alpha_{i,j})&=&a_{1}+ 2a_{2}+ 3a_{3}+ 5a_{4}+ 4a_{5}+ 3a_{6}+ 2a_{7}+ a_{8}-\sum_{s=1}^{i-1}(a_{s+2}+a_{s+3})-\sum_{t=i+3}^{j+1}a_{t}+24-i-j\\
(\lambda+\rho,\alpha_{0,0})&=&a_{1}+ 3a_{2}+ 3a_{3}+ 5a_{4}+ 4a_{5}+ 3a_{6}+ 2a_{7}+ a_{8}+22\\
\end{eqnarray*}
\begin{eqnarray*}
(\lambda+\rho,\alpha_{i_{1},i_{2},i_{3},i_{4}})&=&a_{1}+ a_{2}+ 2a_{3}+ 3a_{4}+ 3a_{5}+ 3a_{6}+ 2a_{7}+ a_{8}-\sum_{q=1}^{i_{1}-1}(a_{q+2}+a_{q+3}+a_{q+4}+a_{q+5})\\
&-&\sum_{r=1}^{i_{2}-i_{1}-1}(a_{r+i_{1}+2}+a_{r+i_{1}+3}+a_{r+i_{1}+4})-(\sum_{s=1}^{i_{3}-i_{2}-1}a_{s+i_{2}+2}+\sum_{t=1}^{i_{4}-i_{2}-2}a_{t+i_{2}+3})\\
&+&26-i_{1}-i_{2}-i_{3}-i_{4}.
\end{eqnarray*}
Using these calculations and by calculating $(\varpi_{k},\alpha)$ for any $\alpha\in\Phi_{k,E_{8}}^{+}$, we can verify that $M_{\lambda,k}^{E_{8}}$ is given by
$$
\left\{
\begin{array}{ll}
a_{1}+ 3a_{2}+ 3a_{3}+ 5a_{4}+ 4a_{5}+ 3a_{6}+ 2a_{7}+ a_{8}+22& {\rm if}\ k=1\\ 
a_{1}+ a_{2}+ 2a_{3}+ 3a_{4}+ 3a_{5}+ 3a_{6}+ 2a_{7}+ a_{8}+16& {\rm if}\ k=2\\ 
a_{1}+ a_{2}+ a_{3}+ 2a_{4}+ 2a_{5}+ 2a_{6}+ 2a_{7}+ a_{8}+12& {\rm if}\ k=3\\ 
a_{1}+ a_{2}+ a_{3}+ a_{4}+ a_{5}+ a_{6}+ a_{7}+ a_{8}+8 & {\rm if}\ k=4\\ 
a_{1}+ a_{2}+ 2a_{3}+ 2a_{4}+ a_{5}+ a_{6}+ a_{7}+ a_{8}+10 & {\rm if}\ k=5\\ 
a_{1}+ 2a_{2}+ 2a_{3}+ 3a_{4}+2a_{5}+ a_{6}+ a_{7}+ a_{8}+13 & {\rm if}\ k=6\\ 
2a_{1}+ 2a_{2}+ 3a_{3}+ 4a_{4}+3a_{5}+ 2a_{6}+ a_{7}+ a_{8}+18 & {\rm if}\ k=7\\ 
2a_{1}+ 3a_{2}+ 4a_{3}+ 6a_{4}+5a_{5}+ 4a_{6}+ 3a_{7}+ 2a_{8}+29 & {\rm if}\ k=8\\ 
\end{array}
.\right.
$$
\subsection{type $F_{4}$}
Let $G$ be a simply connected simple Lie group with a Dynkin diagram of type $F_{4}$, as follows.
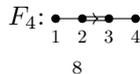
\begin{center}
$F_{4}$:\dynkin[label]{F}{4}
\end{center}

We denote $I^{\prime}_{4}$ by $I_{4}+\mathbb{Z}((\epsilon_{1}+\epsilon_{2}+\epsilon_{3}+\epsilon_{4})/2)$ 
and define $\Phi_{F_{4}}=\{\alpha\in I^{\prime}_{4}\ |\ (\alpha,\alpha)\in \{1,2\}\}$. Then, $\Phi_{F_{4}}$ consists of all $\pm\epsilon_{i}$, all $\pm(\epsilon\pm\epsilon_{j})$, $i\neq j$, and $\pm\frac{1}{2}(\epsilon_{1}\pm\epsilon_{2}\pm\epsilon_{3}\pm\epsilon_{4})$, where the signs may be chosen independently. As a base, we take
$$\Delta_{4}:=\{\alpha_{1}:=\epsilon_{2}-\epsilon_{3},\ \alpha_{2}:=\epsilon_{3}-\epsilon_{4},\ \alpha_{3}:=\epsilon_{4},\ \alpha_{4}:=\frac{1}{2}(\epsilon_{1}-\epsilon_{2}-\epsilon_{3}-\epsilon_{4})\}.$$
The fundamental weights are as follows:
\begin{eqnarray*}
\varpi_{1}&=&\epsilon_{1}+\epsilon_{2}\\
\varpi_{2}&=&2\epsilon_{1}+\epsilon_{2}+\epsilon_{3}\\
\varpi_{3}&=&\frac{1}{2}(3\epsilon_{1}+\epsilon_{2}+\epsilon_{3}+\epsilon_{4})\\
\varpi_{4}&=&\epsilon_{1}.\\
\end{eqnarray*}
The set of positive roots, $\Phi_{F_{4}}^{+}$, is
$$\{\epsilon_{i}\}_{1\leq i \leq 4}\cup\{\epsilon_{i}\pm\epsilon_{j}\}_{1\leq i<j\leq4}\cup\{\alpha_{2,3,4},\alpha_{i,j},\beta_{0},\beta_{i}\}_{2\leq i<j\leq4},$$
where $\beta_{0}$, $\alpha_{2,3,4}$, and $\alpha_{i,j}$ are identical to the notations used for type $E_{8}$, and $\beta_{i}$ denotes the  root such that the coefficient of $\epsilon_{i}$ is $-\frac{1}{2}$ and those of the remaining are $\frac{1}{2}$. The dimension of homogeneous variety $F_{4}/P_{\alpha_{k}}$ is
$$\left\{
\begin{array}{ll}
15 & {\rm if}\ k=1,4\\ 
20& {\rm if}\ k=2,3
\end{array}
.\right.
$$
We now calculate the pairing values of $\lambda+\rho$ with all positive roots. For any $2\leq i<j\leq4$ and $2\leq l\leq4$, we have
\begin{eqnarray*}
(\lambda+\rho,\epsilon_{1})&=&a_{1}+2a_{2}+\frac{3}{2}a_{3}+a_{4}+\frac{11}{2}\\
(\lambda+\rho,\epsilon_{l})&=&a_{1}+a_{2}+\frac{1}{2}a_{3}-\sum_{s=1}^{l-2}a_{s}+\frac{9}{2}-l\\
(\lambda+\rho,\epsilon_{1}+\epsilon_{l})&=&2a_{1}+3a_{2}+2a_{3}+a_{4}-\sum_{s=1}^{l-2}a_{s}+10-l\\
(\lambda+\rho,\epsilon_{i}+\epsilon_{j})&=&a_{1}+2a_{2}+a_{3}-\sum_{s=1}^{j-3}a_{1+s}-\sum_{t=1}^{i-2}a_{t}+9-i-j\\
(\lambda+\rho,-\epsilon_{1}+\epsilon_{l})&=&\sum_{s=1}^{3}a_{s}+\sum_{t=1}^{l-2}a_{t}+l+1\\
(\lambda+\rho,-\epsilon_{2}+\epsilon_{j})&=&\sum_{s=1}^{j-i}a_{s}+j-2\\
(\lambda+\rho,-\epsilon_{3}+\epsilon_{4})&=&a_{2}+1\\
(\lambda+\rho,\beta_{0})&=&a_{1}+2a_{2}+\frac{3}{2}a_{3}+\frac{1}{2}a_{4}+5\\
(\lambda+\rho,\beta_{l})&=&a_{2}+a_{3}+\frac{1}{2}a_{4}+\sum_{t=1}^{i-2}a_{t}+i+\frac{1}{2}\\
(\lambda+\rho,\alpha_{i,j})&=&\frac{1}{2}a_{3}+\frac{1}{2}a_{4}+\sum_{s=1}^{j-3}a_{s+1}+\sum_{t=1}^{i-2}a_{t}+i+j-4\\
(\lambda+\rho,\alpha_{2,3,4})&=&\frac{1}{2}a_{4}+\frac{1}{2}.
\end{eqnarray*}
Using these calculations and by calculating $(\varpi_{k},\alpha)$ for any $\alpha\in\Phi_{k,F_{4}}^{+}$, we can verify that $M_{\lambda,k}^{F_{4}}$ is given by
$$\left\{
\begin{array}{ll}
a_{1}+3a_{2}+2a_{3}+a_{4}+7& {\rm if}\ k=1\\ 
a_{1}+a_{2}+a_{3}+a_{4}+4& {\rm if}\ k=2\\
a_{1}+2a_{2}+a_{3}+a_{4}+5& {\rm if}\ k=3\\
2a_{1}+4a_{2}+3a_{3}+a_{4}+10& {\rm if}\ k=4
\end{array}
.\right.
$$
\subsection{type $G_{2}$}
Let $G$ be a simply connected simple Lie group with a Dynkin diagram of type $G_{2}$, as follows.
\begin{center}
$G_{2}$:\dynkin[label]{G}{2}
\end{center}

Let $V_{G_{2}}$ be a vector subspace of $\mathbb{R}^{3}$ consisting of $\sum_{i=1}^{3}a_{i}\epsilon_{i}$ such that $\sum_{i}^{3}a_{i}=0$. We define $\Phi_{G_{2}}=V_{G_{2}}\cap I_{3}$. As a set, $\Phi_{G_{2}}$ comprises
$$\{\pm(\epsilon_{i}-\epsilon_{j})\ |\ 1\leq i<j\leq3\} \cup \{\pm(2\epsilon_{i}-\epsilon_{j}-\epsilon_{k})\ |\ \{1,2,3\}=\{i,j,k\}\}.$$ 
As a base, we take
$$\Delta_{G_{2}}=\{\alpha_{1}:=\epsilon_{2}-\epsilon_{3},\ \alpha_{2}:=\epsilon_{1}-2\epsilon_{2}+\epsilon_{3}\}.$$
The fundamental weights are as follows:
\begin{eqnarray*}
\varpi_{1}&=&\epsilon_{1}-\epsilon_{3}\\
\varpi_{2}&=&2\epsilon_{1}-\epsilon_{2}-\epsilon_{3}.
\end{eqnarray*}
The set of positive roots, $\Phi_{G_{2}}^{+}$, is
$$\{\epsilon_{i}-\epsilon_{j}\ |\ 1\leq i<j\leq3\}\cup\{-\beta_{1},\beta_{2},\beta_{3}\},$$
where $\beta_{i}\ (1\leq i\leq3)$ is a  root such that the coefficient of $\epsilon_{i}$ is $-2$ and those of the remaining are $1$. The dimension of variety $G_{2}/P_{\alpha_{k}}\ (k=1,2)$ is $5$.
We now calculate the pairing values of $\lambda+\rho$ with all positive roots. For any $1\leq i<j\leq6$, $1\leq i_{1}< i_{2}< i_{3} \leq6$ and $1\leq l\leq6$, we have
\begin{eqnarray*}
(\lambda+\rho,\epsilon_{1}+\epsilon_{2})&=&a_{1}+3a_{2}+4 \\
(\lambda+\rho,\epsilon_{1}+\epsilon_{3})&=&2a_{1}+3a_{2}+5\\
(\lambda+\rho,\epsilon_{2}+\epsilon_{3})&=&a_{1}+1\\
(\lambda+\rho,-\beta_{1})&=&3a_{1}+6a_{2}+9\\
(\lambda+\rho,\beta_{2})&=&3a_{2}+3\\
(\lambda+\rho,\beta_{3})&=&3a_{1}+3a_{2}+6.
\end{eqnarray*}
The set $\Phi_{1,G_{2}}^{+}$ consists of $\epsilon_{i}-\epsilon_{j}$, $-\beta_{1}$, or $\beta_{3}$ for all $1\leq i<j\leq3$. Conversely, the set $\Phi_{2,G_{2}}^{+}$ consists of $\epsilon_{1}-\epsilon_{i}$ or $\beta_{j}$ for all $2\leq i\leq3$ and $1\leq j\leq 3$. We can verify that $M_{\lambda,k}^{G_{2}}$ is given by
$$\left\{
\begin{array}{ll}
a_{1}+ 3a_{2}+4& {\rm if}\ k=1\\ 
a_{1}+a_{2}+2& {\rm if}\ k=2
\end{array}
.\right.
$$
\section{Applications}
In this section, we present certain applications for the main theorem. First, we prove a common claim for varieties of each type and then describe corresponding applications.

\begin{proof}[Proof of Corollary \ref{Cor}]Let $E_{\lambda}$ be an irreducible homogeneous vector bundle with $\lambda=\sum_{i=1}^{n}a_{i}\varpi_{i}$. We assume that $a_{k}=0$, without the loss of generality. As the number of elements in $T_{\lambda,k}$, which is not equal to zero, is equal to dim $G/P_{\alpha_{k}}$, if $M^{G}_{\lambda,k}>{\rm dim}\ G/P_{\alpha_{k}}$, there exists an integer $l \in[1,M^{G}_{\lambda,k}]$ such that $n_{l}=0$. By the main theorem, $E_{\lambda}$ is not an ACM bundle. Therefore, if $E_{\lambda}$ is an ACM bundle, then $M^{G}_{\lambda,k}\leq {\rm dim}\ G/P_{\alpha_{k}}$. As $M^{G}_{\lambda,k}$ is a linear combination of $a_{i}$ and $a_{i}\geq0$, there exist only a finite number of choices for $a_{i}$. Therefore, there exist only a finite number of irreducible homogeneous ACM bundles.
\end{proof}

In the subsequent subsections, we decide on the choice of $a_{i}$ for particular varieties.
\subsection{For the Cayley Planes}
Using Theorem \ref{Thm}, we determine initialized irreducible homogeneous ACM bundles on the Cayley Plane $E_{6}/P_{\alpha_{1}}\cong E_{6}/P_{\alpha_{6}}$.
\begin{cor}\label{Cor2} Let $E_{\lambda}$ be an initialized irreducible homogeneous vector bundle over the Cayley Plane $E_{6}/P(\alpha_{1})$ with $\lambda=\sum_{i=1}^{6}a_{i}\varpi_{i}$. Then, $E_{\lambda}$ is an ACM bundle if and only if $(a_{1},a_{2},a_{3},a_{4},a_{5},a_{6})$ is equal to equals $(0,0,0,0,i,j)$, where $i=0,1$ and $j=0,1,2,3$.
\end{cor}
\begin{proof}
By Lemma \ref{lem}, we may assume that $a_{1}=0$. We proceed using the following four steps.

Step $1.$ We prove that $a_{3}\geq1$ if and only if $E_{\lambda}$ is not an ACM bundle. As per Theorem \ref{Thm}, proving that $a_{3}\geq1$ if and only if integer $l \in[1,2a_{2}+2a_{3}+3a_{4}+2a_{5}+a_{6}+11]$ exists such that $n_{l}=0$ is sufficient. Based on $a_{i}\geq 0$ for $2\leq i \leq 6$, note that $2a_{2}+2a_{3}+3a_{4}+2a_{5}+a_{6}+11\geq 11$. In particular, $2$ is contained in $[1,2a_{2}+2a_{3}+3a_{4}+2a_{5}+a_{6}+11]$. In addition, note that $a_{3}+2$ has degree one with respect to variables $a_{i}$ in $T_{\lambda,1}^{E_{6}}$ and the smallest constant term  (henceforth, such an element is called the minimum element with respect to $a_{i}$). Thus, $a_{3}\geq1$ if and only if $n_{2}=0$.

Step $2.$ Based on step $1$, for $E_{\lambda}$ to be an ACM bundle, we assume that $a_{3}=0$. We then prove that $a_{4}\geq 1$ if and only if $E_{\lambda}$ is not an ACM bundle. As the minimum element with respect to $a_{i}$ in $T_{\lambda,1}^{E_{6}}$ is $a_{4}+3$, the argument used in step $1$ implies that $a_{4}\geq1$ if and only if $n_{3}=0$.

Step $3.$ Based on steps $1$ and $2$, let $a_{3}$ and $a_{4}$ be zero. In this step, we prove that $a_{2}\geq 1$ and $a_{5}\geq 1$ if and only if $E_{\lambda}$ is not an ACM bundle. The minimum element with respect to $a_{i}$ in $T_{\lambda,1}^{E_{6}}$ is either $a_{5}+4$ or $a_{2}+4$. Therefore, we have $a_{2}\geq 1$ and $a_{5}\geq 1$ if and only if $n_{4}=0$.

Step $4.$ Based on the three previous steps, considering the following cases is sufficient for $E_{\lambda}$ to be an ACM bundle.

Case $1.$ $a_{1}=a_{2}=a_{3}=a_{4}=0$. In this case, we obtain $a_{5}\geq2$ if and only if $n_{5}=0$. If $a_{5}=i$ ($i=0$ or $1$), we obtain $a_{6}\geq4$ if and only if $n_{8+i}=0$. Conversely, under this condition, if $0\leq a_{6}\leq 3$, then $n_{l}\geq1$ for all $l\in[1,11+2i+a_{6}]$.

Case $2.$ $a_{1}=a_{3}=a_{4}=a_{5}=0$. In this case, we have $a_{2}\geq2$ and $a_{6}\geq1$ if and only if $n_{5}=0$. The case where $a_{2}$ is zero has just been considered. If $a_{2}$ is one, $M_{\lambda,1}^{E_{6}}$ is equal to $a_{6}+13$. However, the element in $T_{\lambda,1}^{E_{6}}$ with the second-largest constant term after $M_{\lambda,1}^{E_{6}}$ is $a_{6}+11$. As all coefficients of $a_{6}$ are equal, there is no nonnegative integer, $a_{6}$, satisfying condition $n_{l}\geq1$ for all $l\in[1,a_{6}+13]$. Finally, we consider the case in which $a_{6}=0$ under this condition. As the coefficients of $M_{\lambda,1}^{E_{6}}$, which are equal to $2a_{2}+11$ and $a_{2}+10$, are different and all coefficients of $a_{2}$ in $T_{\lambda,1}^{E_{6}}$ except $M_{\lambda,1}^{E_{6}}$ are identical, if $a_{2}\geq1$ then there exists integer $l\in[1,2a_{2}+11]$ such that $n_{l}=0$. Therefore, in this case, $E_{\lambda}$ is an ACM bundle if and only if $a_{2}=a_{6}=0$.
\end{proof}
\subsection{For the Freudenthal varieties}
We consider the Freudenthal variety, $E_{7}/P_{\alpha_{7}}$. We derive a condition for the highest weight for an irreducible homogeneous vector bundle on this variety to be an ACM bundle.
\begin{cor} Let $E_{\lambda}$ be an initialized irreducible homogeneous vector bundle over the Freudenthal variety $E_{7}/P_{\alpha_{7}}$ with $\lambda=\sum_{i=1}^{7}a_{i}\varpi_{i}$. Then, $E_{\lambda}$ is an ACM bundle if and only if $(a_{1},a_{2},a_{3},a_{4},a_{5},a_{6},a_{7})$ is equal to $(0,i,0,0,0,0,0)$, where $i=0,1,2$.
\end{cor}
\begin{proof}
As in the case of Corollary \ref{Cor2}, we assume that $a_{4}=a_{5}=a_{6}=a_{7}=0$ for $E_{\lambda}$ to be an ACM bundle. Then, we obtain $a_{2}\geq 1$ and $a_{3} \geq 1$ if and only if $n_{5}=0$. Therefore, we consider two cases: $a_{2}=0$ and $a_{3}=0$.

When $a_{2}=0$, we have $a_{3}\geq 2$ if and only if $n_{6}=0$. For $E_{\lambda}$ to be an ACM bundle, we consider the cases $a_{3}=0$ and $a_{3}=1$. First, when $a_{3}=1$, $T^{E_{7}}_{\lambda,7}$ consists of $i$, $a_{1}+j$ or $2a_{1}+20$, where $1\leq i\leq 10$ and $j\in\{7,8,\cdots,17,19\}$. In other words, $a_{1}+18$ does not belong to $T^{E_{7}}_{\lambda,7}$. Therefore, there exists no nonnegative integer $a_{1}$ such that $E_{\lambda}$ is an ACM bundle in this case. When $a_{3}=0$, $T^{E_{7}}_{\lambda,7}$ consists of $i$, $a_{1}+j$, or $2a_{1}+17$, where $1\leq i\leq 9$ and $6\leq j\leq 16$. $M^{E_{7}}_{\lambda,7}$ is $2a_{1}+17$. As $M^{E_{7}}_{\lambda,7}$ and other elements in $T^{E_{7}}_{\lambda,7}$ have different coefficients of $a_{1}$ and all coefficients of $a_{1}$ in $T^{E_{7}}_{\lambda,1}$ except $M^{E_{7}}_{\lambda,7}$ are identical, if $a_{1}\geq 1$, then there exists an integer $l \in[1,M_{\lambda,k}]$ such that $n_{l}=0$. Conversely, if $a_{1}$ is equal to zero, $n_{l}\geq1$ for any $l\in[1,17]$.

We consider the case when $a_{3}=0$. Then, we obtain $a_{2}\geq 2$ and $a_{1} \geq 1$ if and only if $n_{6}=0$. For $E_{\lambda}$ to be ACM bundle, considering the cases $a_{2}=0$, $a_{2}=1$, and $a_{1}=0$ is sufficient. For $a_{2}=0$ and $a_{2}$=1, $a_{1}$ has only zero  by a similar argument to that in the previous case. When $a_{1}=0$, we have $a_{2} \geq 3$ if and only if $n_{7}=0$. Then, $E_{\lambda}$ is an ACM bundle if $a_{2}$ is $0, 1$, or $2$.

Therefore, $E_{\lambda}$ is an ACM bundle if and only if $a_{i}=0$ for $i\neq2$ and $a_{2}$ is $0, 1$, or $2$.
\end{proof}
\subsection{For type $E_{8}$}
In this subsection, we classify ACM bundles on homogeneous varieties of type $E_{8}$. In particular, we focus on the case where the parabolic subgroup is defined by $\alpha_{8}$.
\begin{cor} Let $E_{\lambda}$ be an initialized irreducible homogeneous vector bundle over $E_{8}/P_{\alpha_{8}}$ with $\lambda=\sum_{i=1}^{8}a_{i}\varpi_{i}$. Then $E_{\lambda}$ is ACM bundle if and only if $(a_{1},a_{2},a_{3},a_{4},a_{5},a_{6},a_{7},a_{8})$ is equal to $(i,0,0,0,0,0,0,0)$, $(i,0,1,0,0,0,0,0)$, $(j,1,0,0,0,0,0,0)$, or $(0,k,0,0,0,0,0,0)$, where $0\leq i\leq 5$, $1\leq j\leq4$ and $1\leq k \leq 2$.
\end{cor}
\begin{proof}
As in Corollary \ref{Cor2}, we assume that $a_{8}=a_{7}=a_{6}=a_{5}=a_{4}=0$ for $E_{\lambda}$ to be an ACM bundle. Then, we obtain $a_{2}\geq1$ and $a_{3}\geq1$ if and only if $n_{6}=0$. Thus, it is sufficient to consider the following cases.

Firstly, we consider the case when $a_{3}=0$. Then, we have $a_{2}\geq2$ and $a_{1}\geq1$ if and only if $n_{7}=0$. 
\begin{itemize}
\item If $a_{2}=0$, then $a_{1}\geq6$ if and only if $n_{12}=0$. If $0\leq a_{1}\leq 5$, we can confirm that $n_{l}\geq1$ for all $l\in [1,M_{\lambda,8}^{E_{8}}]$. 
\item In the case when $a_{2}=1$, $a_{1}\geq7$ if and only if $n_{13}=0$. However in this case, if $a_{1}=5+i$ ($i=0$ or $1$) then $n_{29+i}=0$. On the other hand, if $0\leq a_{1}\leq4$, then $n_{l}\geq1$ for all $l\in [1,M_{\lambda,8}^{E_{8}}]$.
\item When $a_{1}=0$, $a_{2}\geq3$ if and only if $n_{8}=0$. If $0\leq a_{2}\leq 2$, we can verify that $n_{l}\geq1$ for all $l\in [1,M_{\lambda,8}^{E_{8}}]$. 
\end{itemize}

In the case when $a_{2}=0$, we obtain $a_{3}\geq2$ if and only if $n_{7}=0$. We have already considered the case when $a_{3}$ is equal to zero. Therefore, we now only consider the case in which $a_{3}$ is equal to one. Thus, we have $a_{1}\geq6$ if and only if $n_{13}=0$. If $0\leq a_{1}\leq5$, then we can verify that $n_{l}\geq1$ for all $l\in [1,M_{\lambda,8}^{E_{8}}]$.
\end{proof}
\subsection{For type $F_{4}$}
In this subsection, we classify ACM bundles on homogeneous varieties of type $F_{4}$. In particular, we consider the case in which a parabolic subgroup is defined by $\alpha_{1}$. We are able to classify initialized, irreducible, homogeneous ACM bundles on this variety.
\begin{cor} Let $E_{\lambda}$ be an initialized irreducible homogeneous vector bundle over $F_{4}/P_{\alpha_{1}}$ with $\lambda=\sum_{i=1}^{4}a_{i}\varpi_{i}$. $E_{\lambda}$ is an ACM bundle if and only if $(a_{1},a_{2},a_{3},a_{4})$ is equal to $(0,0,0,i)$ or $(0,0,1,j)$, where $0\leq i\leq4,\ j=0,1,2,3,5$.
\end{cor}
\begin{proof}
For $E_{\lambda}$ to be an ACM bundle, we may assume that $a_{1}=a_{2}=0$. Then, $a_{3}\geq 2$ if and only if $n_{3}=0$. Thus, considering the cases of $a_{3}=0$ and $a_{3}=1$ is sufficient.

When $a_{3}=0$, we have $a_{4}\geq5$ if and only if $n_{5}=0$. If $a_{4}$ belongs to the set $[0,4]$, then $n_{l}\geq1$ for all $l\in[1,a_{4}+7]$.

For $a_{3}=1$, we obtain $a_{4}\geq6$ if and only if $n_{6}=0$. If $a_{4}$ is $4$, $n_{6}$ is equal to zero. Conversely, if $a_{4}$ is equal to $0$, $1$, $2$, $3$, or $5$, we obtain $n_{l}\geq1$ for $l\in[1,a_{4}+9]$.
\end{proof}
\subsection{For type $G_{2}$}
Finally, we consider homogeneous varieties of type $G_{2}$. In these cases, we can obtain simple formulas for an initialized irreducible homogeneous vector bundle on such varieties to be an ACM bundle. 
\begin{cor} Let $E_{\lambda}$ be an initialized irreducible homogeneous vector bundle over $G_{2}/P_{\alpha_{1}}$ with $\lambda=a_{1}\varpi_{1}+a_{2}\varpi_{2}$. Then, $E_{\lambda}$ is an ACM vector bundle if and only if $a_{1}=a_{2}=0$.
\end{cor}
\begin{proof}
Based on initialization (i.e., $a_{1}=0$), set $T^{G_{2}}_{\lambda,1}$ is $\{3a_{2}+4,\ \frac{1}{2}(3a_{2}+5),\ 2a_{2}+3,\ 1,\ a_{2}+2\}$.
If $a_{2}\geq 1$, then $3a_{2}+4 \geq 7$. As there are only five equations in $T^{G_{2}}_{\lambda,1}$, there exists an integer $l \in [1,3a_{2}+4]$ such that $n_{l}=0$. By Theorem \ref{Thm}, $E_{\lambda}$ is not an ACM vector bundle. Therefore, $E_{\lambda}$ is an ACM vector bundle if and only if $a_{1}=a_{2}=0$.
\end{proof}

\begin{cor}\ Let $E_{\lambda}$ be an initialized irreducible homogeneous vector bundle over $G_{2}/P_{\alpha_{2}}$ with $\lambda=a_{1}\varpi_{1}+a_{2}\varpi_{2}$. Then, $E_{\lambda}$ is an ACM vector bundle if and only if $(a_{1},a_{2})$ is equal to $(i,0)$, where $i=0,1,2$.
\end{cor}
\begin{proof}
The set $T^{G_{2}}_{\lambda,2}$ is given by $\{ \frac{1}{3}(a_{1}+4),\ \frac{1}{3}(2a_{1}+5),\ \frac{1}{6}(3a_{1}+9),\ 1,\ a_{1}+2\}$. If $a_{1}>3$, then $a_{1}+2 \geq 6$. As there are only five equations in $T^{G_{2}}_{\lambda,1}$, there exists an integer $l \in [1,a_{1}+2]$ such that $n_{l}=0$. When $a_{1}=3$, we have $\frac{1}{3}(a_{1}+4)=\frac{7}{3}$. Therefore, there exists an integer $l \in [1,5]$ such that $n_{l}=0$. In these case, as per Theorem \ref{Thm}, $E_{\lambda}$ is not an ACM vector bundle.

Suppose that $a_{1}=1$. Set $T^{G_{2}}_{\lambda,2}$ is given by $\{\frac{5}{3},\ \frac{7}{3},\ 2,\ 1,\ 3\}$. Then, $n_{l}=1$ for $l \in [1,3]$. Again, as per Theorem \ref{Thm}, $E_{\lambda}$ is an ACM vector bundle. When $a_{1}$ is equal to $2$, then set $T_{\lambda,2}^{G_{2}}$ is given by $\{1,2,3,4,\frac{5}{2}\}$. Hence, $E_{\lambda}$ is an ACM vector bundle if and only if $a_{2}=0$ and $a_{1}=0,1,2$.
\end{proof}


\begin{thebibliography}{MMMN}
\bibitem{BHP}E. Ballico, S. Huh, and J. Pons-Llopis, aCM\ vector\ bundles\ on\ projective\ surfaces\ of\ nonnegative\ Kodaira\ dimension, {\it International\ Journal\ of\ Mathematics\/}, $32(14)$:$2150109, 2021.$
\bibitem{BF}M. C. Brambilla and D. Faenzi, Moduli spaces of rank-2 ACM bundles on prime Fano threefolds, {\it Michigan\ Mathematical\ Journal\/}, $60(1)$:$113-148, 2011.$
\bibitem{CH}M. Casanellas and R. Hartshorne, ACM bundles on cubic surfaces, {\it Journal\ of\ the\ European\ Mathematical\ Society\/}, $13(3)$:$709-731, 2011.$
\bibitem{CFM}G. Casnati, D. Faenzi, and F. Malaspina, Rank two aCM bundles on the del Pezzo threefold with Picard number $3$, {\it Journal\ of\ Algebra\/}, $429$:$413-446, 2015.$
\bibitem{CN}G. Casnati and R. Notari, Examples of rank two aCM bundles on smooth quartic surfaces in $\mathbb{P}^{3}$, {\it Rendiconti\ del\ Circolo\ Matematico\ di\ Palermo}, Serie $II, 66(1)$:$19-41, 2017.$
\bibitem{CM}L. Costa and R. Maria Mir\'{o}-Roig, Homogeneous ACM bundles on a Grassmannian, {\it Advances\ in\ Mathematics}, $289$:$95-113, 2016.$
\bibitem{Dem}M. Demazure, A very simple proof of Bott's theorem, {\it Inventiones mathematicae\/}, $33,\ 271$-$272$, $1976.$
\bibitem{DFR}R. Du, X. Fang, and P. Ren, Homogeneous ACM bundles on isotropic Grassmannians, arXiv:$2206.09172$v$1$.
\bibitem{Fae}D. Faenzi, Rank $2$ arithmetically Cohen-Macaulay bundles on a nonsingular cubic surface, {\it Journal\ of\ Algebra\/}, $319(1)$:$143-186, 2008.$
\bibitem{Fil}M. Filip, Rank $2 $ACM bundles on complete intersection Calabi-Yau threefolds, {\it Geometriae\ Dedicata\/}, $173$:$331-346, 2014.$
\bibitem{Fon}A. Fonarev, Irreducible Ulrich bundles on isotropic Grassmannians, {\it Moscow Mathematical Journal}, $16$ ,no. $4$, $711-726$. $2016$. 
\bibitem{Hor}G. Horrocks, Vector bundles on the punctured spectrum of local ring, {\it Proceedings\ of\ the\ London\ Mathematical\ Society\/}, $3(4)$:$689-713$, $1964.$
\bibitem{Hum}J. E. Humphreys, Introduction to Lie Algebras and Representation Theory, {\it Graduate\ Texts\ in\ Mathematics\/}, $9.$\ Springer,\ New\ York,\ $1978.$  
\bibitem{Jan} J. C. Jantzen, Representations of algebraic groups, Second edition. Mathematical Surveys and Monographs, $107$. {\it American\ Mathematical\ Society\/}, Providence, RI, $2003. $
\bibitem{LP}Kyoung-Seog Lee and Kyeong-Dong Park, Equivariant Ulrich bundles on exceptional homogeneous varieties, {\it Advances in Geometry\/}, $21$ no.$2$, $187-205$. $2021$.
\bibitem{Ott}G. Ottaviani, Rational homogeneous varieties. {\it Lecture\ notes\ for\ the\ summer\ school\ in Algebraic\ Geometry\ in\ Cortona\/},\ $1995.$
\bibitem{RT}G. V. Ravindra and A. Tripathi, Rank 3 ACM bundles on general hypersurfaces in $\mathbb{P}^{5}$, {\it Advances\ in\ Mathematics\/}, $355$:$106780,\ 2019.$
\bibitem{Wat}K. Watanabe, ACM line bundles on polarized K$3$ surfaces, {\it Geometriae\ Dedicata\/}, $203(1)$:$321-335, 2019.$
\bibitem{Yos}K. Yoshioka, aCM bundles on a general abelian surface, {\it Archiv\ der\ Mathematik}, $116(5)$:$529-539, 2021.$ 
\end{thebibliography}
\end{document}